\documentclass[12pt,a4paper]{article}
\usepackage{amssymb,amsmath}
\usepackage{mathrsfs}
\usepackage{tikz}
\usepackage{graphicx}
\usepackage{amsthm}

\theoremstyle{definition} 
\newtheorem{df}{Definition}[section] 
   
\theoremstyle{plain}            
\newtheorem{pro}[df]{Proposition}
\newtheorem{lem}[df]{Lemma}
\newtheorem{theo}[df]{Theorem}
\newtheorem{cor}[df]{Corollary}

\newcommand{\al}{\ensuremath{\alpha}}

\newcommand{\la}{\ensuremath{\lambda}}

\newcommand{\Bcal}{\ensuremath{\mathcal{B}}}

\newcommand{\Ecal}{\ensuremath{\mathcal{E}}}
\newcommand{\Fcal}{\ensuremath{\mathcal{F}}}
\newcommand{\Hcal}{\ensuremath{\mathcal{H}}}
\newcommand{\Mcal}{\ensuremath{\mathcal{M}}}
\newcommand{\Ncal}{\ensuremath{\mathcal{N}}}

\newcommand{\nn}{\ensuremath{\mathbb{N}}}

\newcommand{\rr}{\ensuremath{\mathbb{R}}}

\newcommand{\unit}{\ensuremath{\mathbf{1}}}
\newcommand{\norm}[1]{\ensuremath{\left\|#1\right\|}}

\DeclareMathOperator*{\slim}{s-lim} 
\newcommand{\set}[2]{\left\{#1\,\middle|\,#2 \right\}}

\newcommand{\Ca}{$C${\rm*}-algebra}      
\newcommand{\Csa}{$C${\rm*}-subalgebra}

\newcommand{\vNa}{von Neumann algebra}

\newcommand{\AW}{$AW${\rm*}-algebra}
\newcommand{\Jbw}{JBW-algebra}

\newcommand{\Sam}{{\rm*}-automorphism}
\newcommand{\Sic}{{\rm*}-isomorphic}

\newcommand{\ifff}{if and only if}

\begin{document}

\begin{center}
{\Large \bf Vigier's theorem for the spectral order and its applications}\\

{\large Martin Bohata\footnote{bohata@math.feld.cvut.cz}\\}
         \it Department of Mathematics, Faculty of Electrical Engineering,\\
        Czech Technical University in Prague, Technick\'a 2,\\ 
        166 27 Prague 6, Czech Republic        
\end{center}

{\small \textbf{Abstract:} The paper mainly deals with suprema and infima of self-adjoint operators in a \vNa{} \Mcal{} with respect to the spectral order. Let $\Mcal_{sa}$ be the self-adjoint part of \Mcal{} and let $\preceq$ be the spectral order on $\Mcal_{sa}$. We show that a decreasing net in $(\Mcal_{sa},\preceq)$ with a lower bound has the infimum equal to the strong operator limit. The similar statement is proved for an increasing net bounded above in $(\Mcal_{sa},\preceq)$. This version of Vigier's theorem for the spectral order is used to describe suprema and infima of nonempty bounded sets of self-adjoint operators in terms of the strong operator limit and operator means. As an application of our results on suprema and infima, we study the order topology on $\Mcal_{sa}$ with respect to the spectral order. We show that it is finer than the restriction of the Mackey topology.} 


\section{Introduction}

Let \Mcal{} be a von Neumann algebra (i.e. strongly operator closed \Csa{} of the \Ca{} \Bcal(\Hcal) of all bounded operators on a complex Hilbert space \Hcal). Denote by $\Mcal_{sa}$ the self-adjoint part of \Mcal{} and consider the poset $(\Mcal_{sa},\leq)$, where $\leq$ is the standard order given by the positive cone $\Mcal_{+}$ of \Mcal{}. A well known result of Sherman \cite{Sh51} shows that $(\Mcal_{sa},\leq)$ is a lattice \ifff{} \Mcal{} is abelian. In the strongly noncommutative case, $(\Mcal_{sa},\leq)$ is far from being a lattice. More concretely, it was proved by Kadison \cite{Ka51} that \Mcal{} is a factor \ifff{} two self-adjoint elements in \Mcal{} are comparable whenever their infimum exists. 

To obtain a lattice structure, Olson \cite{Ol71} introduced another partial order on $\Mcal_{sa}$ called the spectral order. He defined it in terms of spectral families. Recall that a family $(E_\la)_{\la\in\rr}$ of projections in $\Mcal$ is called a ({\it bounded}) {\it spectral family} (or a {\it bounded resolution of the identity}) if the following conditions hold:
\begin{enumerate}
	\item $E_\la\leq E_\mu$ whenever $\la\leq \mu$.
	\item $E_\la=\inf_{\mu>\la}E_\mu$ for every $\la\in\rr$.
	\item There is a positive real number $\al$ such that $E_\la=0$ when $\la<-\al$ and $E_\la=\unit$ when $\la>\al$, where \unit{} is the unit of \Mcal{}.
\end{enumerate}
By Spectral theorem (see, for example, \cite{Ka97I,Sch12}), there is a bijection between $\Mcal_{sa}$ and the set of all spectral families in $\Mcal$. In the sequel, we shall denote the spectral family of a self-adjoint element $x$ by $(E_\la^x)_{\la\in\rr}$. 
The {\it spectral order} is a partial order $\preceq$ on $\Mcal_{sa}$ defined as follows: $x\preceq y$ if $E_\la^y\leq E_\la^x$ for each $\la\in\rr$. It was proved in \cite{Ol71} that the poset $(\Mcal_{sa},\preceq)$ is a conditionally complete lattice (i.e. a lattice in which every nonempty bounded subset has the infimum and the supremum) and, for each two commuting elements $x,y\in\Mcal_{sa}$, $x\preceq y$ \ifff{} $x\leq y$. Bearing in mind the Sherman result mentioned above, one can immediately observe that $\preceq$ coincides with $\leq$ on $\Mcal_{sa}$ \ifff{} \Mcal{} is abelian.

The spectral order has been intensively studied over the last several decades. It was proved in \cite{Ka79} that if $k\in\nn$ and $x_1,\dots, x_k$ are positive operators, then the supremum of $x_1,\dots, x_k$ with respect to the spectral order is
$$
x_1\vee\dots \vee x_k=\slim_{n\to\infty}\left(x_1^n+\dots +x_k^n\right)^{\frac{1}{n}}
$$
where $\slim$ denotes the strong operator limit. It is obvious that the right-hand side can be expressed as the strong operator limit of the $n$th root of the arithmetic mean of $x_1^n,\dots, x_k^n$. On the other hand, the infimum of two positive invertible operators with respect to the spectral order was described in \cite{An89,FNN98} by using the strong operator limit and the harmonic mean.
More concretely, it was shown that the infimum of two positive invertible operators $x$ and $y$ is
$$
x\wedge y=\slim_{n\to\infty}\left(\frac{x^{-n}+y^{-n}}{2}\right)^{-\frac{1}{n}}.
$$
The formula for the supremum of two positive operators was also obtained in \cite{AW96} as a consequence of a connection between the spectral order majorant of a positive operator and a conditional expectation. Furthermore, there is a close relationship of the spectral order with some other interesting partial orders. It was shown in \cite{FF02} that the complexity order is nothing but the spectral order on positive definite matrices and the spectral order is stronger than entropy order. Moreover, the set $\Ecal(\Mcal)$ of all effects in \Mcal{} (i.e. positive operators in the unit ball of \Mcal{}) endowed with the spectral order forms a complete lattice which was investigated in \cite{dG05}. Bijections on $\Ecal(\Bcal(\Hcal))$ preserving the spectral order were described in \cite{MN16, MS07}. The definition of the spectral order can be naturally extended to different settings going beyond the self-adjoint part of a \vNa{} like unbounded operators \cite{PS12,Tu14}, \AW{}s \cite{HT16,HT17}, and \Jbw{}s \cite{Ha07}. Note that the spectral order on (generally unbounded) positive operators arises in the study of one parametr groups of \Sam{}s of a \vNa{} \cite{Ar74}. The spectral order also plays an important role in physics especially in the topos approach to quantum theory (see, for example, \cite{DD14,Ha11,Wo14}).

The main goal of this paper is to study suprema and infima with respect to the spectral order. In particular, we establish an analogue of Vigier's theorem for the spectral order. Recall that the classical Vigier's theorem says that every decreasing (resp. increasing) net in the poset $(\Mcal_{sa},\leq)$ with a lower (resp. upper) bound has the infimum (resp. supremum) equal to its strong operator limit. We prove that the same is true if we consider the poset $(\Mcal_{sa},\preceq)$ in place of $(\Mcal_{sa},\leq)$. This is a surprising fact because $(\Mcal_{sa},\preceq)$ is far from $(\Mcal_{sa},\leq)$ in the non-commutative case. As a consequence of Vigier's theorem for the spectral order, we obtain a description of infima and suprema of nonempty (not necessarily finite) sets of self-adjoint operators in terms of the strong operator limit and operator means. In particular, we show that if a nonempty set $M$ of positive operators is bounded in $(\Mcal_{sa},\preceq)$ and \Fcal{} is the set of all nonempty finite subsets of $M$, then
$$
	\sup_{x\in M}x=\slim_{F\in\Fcal}\slim_{n\in\nn}\left(\sum_{x\in F}x^n\right)^{\frac{1}{n}}=\slim_{F\in\Fcal}\slim_{n\in\nn}\left(\sum_{x\in F}\frac{x^n}{|F|}\right)^{\frac{1}{n}}.
$$
If, in addition, $M$ has an invertible positive lower bound, then
	$$
	\inf_{x\in M}x=\slim_{F\in\Fcal}\slim_{n\in\nn}\left(\sum_{x\in F}x^{-n}\right)^{-\frac{1}{n}}=\slim_{F\in\Fcal}\slim_{n\in\nn}\left(\sum_{x\in F}\frac{x^{-n}}{|F|}\right)^{-\frac{1}{n}}.
	$$
These formulas for the supremum and the infimum of $M$ in $(\Mcal_{sa},\preceq)$ generalize results from \cite{An89,FNN98,Ka79}.
We also prove that the supremum of mutually orthogonal self-adjoint operators is the sum of their positive parts converging in the strong operator topology. 

The last part of the paper is devoted to the study of the order topology on certain subsets of a \vNa{} endowed with the spectral order. 
Recall that the order topology is defined by using the concept of order convergence introduced by Birkhoff \cite{Bi35}. A net $(x_\al)_{\al\in\Lambda}$ in a poset $(P,\leq)$ is {\it order convergent} to $x\in P$ if there are an increasing net $(y_\al)_{\al\in\Lambda}$ and a decreasing net $(z_\al)_{\al\in\Lambda}$ in $(P,\leq)$ such that $y_\al\leq x_\al\leq z_\al$ for all $\al\in \Lambda$ and $\sup_{\al\in\Lambda}y_\al=\inf_{\al\in\Lambda}z_\al=x$. A set	$M\subseteq P$ is said to be {\it order closed} if no net in $M$ is order convergent to a point in $P\setminus M$. Complements of all order closed sets form a topology $\tau_o(P,\leq)$ on $P$ which is called the {\it order topology}. We say that a topology $\tau$ on $P$ {\it preserves order convergence} if every net order converging to an element $x\in P$ converges to $x$ in $\tau$. It is easy to see that the order topology $\tau_o(P,\leq)$ can be characterized as the finest topology preserving order convergence. 

The order topology on a poset is not Hausdorff in general (see, for example, \cite{FK54}). However, the Hausdorff property holds for order topologies on various subsets of a \vNa{} \Mcal{} endowed with the standard order $\leq$ (given by the positive cone) or the star order. This follows from their close connection with usual locally convex topologies on \Mcal{}. For example, it was proved in \cite{ChHW15} that the order topology $\tau_o(\Mcal_{sa},\leq)$ is finer than the restriction of Mackey topology to $\Mcal_{sa}$. Furthermore, the order topology on $\Mcal{}$ with respect to the star order is finer than $\sigma$-strong* topology \cite{Bo18}. In this paper, we prove that $\tau_o(\Mcal_{sa},\preceq)$ is finer than $\tau_o(\Mcal_{sa},\leq)$ and so it is finer than the restriction of Mackey topology to $\Mcal_{sa}$. This result leads us to a natural question whether or not $\tau_o(\Mcal_{sa},\preceq)$ coincides with $\tau_o(\Mcal_{sa},\leq)$. It is obvious that $\tau_o(\Mcal_{sa},\preceq)=\tau_o(\Mcal_{sa},\leq)$ whenever $\Mcal$ is abelian, because partial orders $\preceq$ and $\leq$ are same in abelian \vNa{}s. In order to give an answer in the non-abelian case, we study restrictions of $\tau_o(\Mcal_{sa},\preceq)$ to some subsets of $\Mcal_{sa}$. In particular, we obtain that $\tau_o(\Mcal_{sa},\preceq)|_{P(\Mcal)}=\tau_o(P(\Mcal),\preceq)$, where $P(\Mcal)$ is the set of all projections in \Mcal{}. Since $\tau_o(\Mcal_{sa},\leq)|_{P(\Mcal)}=\tau_o(P(\Mcal),\leq)$ is true only for abelian \vNa{}s \cite{ChHW15}, we see that $\tau_o(\Mcal_{sa},\preceq)$ and $\tau_o(\Mcal_{sa},\leq)$ coincide \ifff{} $\Mcal$ is abelian.
\section{Infima and suprema}

Throughout this paper, $\Mcal_{sa}$ is the self-adjoint part of a \vNa{} \Mcal{}, $\Mcal_+$ is the positive part of \Mcal{}, $B_1(\Mcal_{sa})$ is the unit ball of $\Mcal_{sa}$, $\Ecal(\Mcal)=\Mcal_{+}\cap B_1(\Mcal_{sa})$ is the set of all effects in \Mcal{}, and $P(\Mcal)$ is the set of all projections in \Mcal{}.

In the following proposition, we summarize some well known results on the spectral order which will be useful in the sequel. 

\begin{pro}[\cite{FK71,dG05,Ol71}]\label{Well known results}
	Let \Mcal{} be a \vNa{}. Assume that $x,y\in \Mcal_{sa}$, $\al,\beta\in\rr$, and $\al>0$.
		\begin{enumerate}
			\item If $x\preceq y$, then $x\leq y$.
			\item $x\preceq y$ \ifff{} $f(x)\leq f(y)$ for any continuous increasing function $f:\rr\to\rr$.
			\item If $M\subseteq\Mcal_{sa}$ is nonempty and bounded above, then the spectral family of the supremum of $M$ in $(\Mcal_{sa},\preceq)$ is $(\inf_{x\in M} E^{x}_\la)_{\la\in\rr}$, where the infimum of projections is considered in the projection lattice $(P(\Mcal),\leq)$.
			\item If $M\subseteq\Mcal_{sa}$ is nonempty and bounded below, then the spectral family of the infimum of $M$ in $(\Mcal_{sa},\preceq)$ is $(\inf_{\mu>\la}\sup_{x\in M} E^{x}_\mu)_{\la\in\rr}$, where the supremum and the infimum of projections is considered in the projection lattice $(P(\Mcal),\leq)$.
			\item $(P(\Mcal),\preceq)=(P(\Mcal),\leq)$.
			\item $x\preceq y$ \ifff{} $\al x+\beta\unit\preceq\al y+\beta\unit$.
		\end{enumerate}
\end{pro}

\begin{pro}\label{suprema and infima in various posets}
Let \Mcal{} be a \vNa{}.
	\begin{enumerate}
		\item If $M\subseteq \Mcal_+$ is nonempty and bounded above in $(\Mcal_{sa},\preceq)$, then the supremum of $M$ in $(\Mcal_{sa},\preceq)$ is a positive element. 
		\item If $M\subseteq \Mcal_+$ is nonempty, then the infimum of $M$ in $(\Mcal_{sa},\preceq)$ is a positive element.
		\item If $L\in\left\{B_1(\Mcal_{sa}),\Ecal(\Mcal),P(\Mcal)\right\}$, then the supremum and the infimum of every nonempty subset of $L$ in $(\Mcal_{sa},\preceq)$ belong to $L$.
	\end{enumerate}
\end{pro}
	\begin{proof}\hfill{}
		\begin{enumerate}
			\item Let $y$ be the supremum of $M$ in $(\Mcal_{sa},\preceq)$. Take $x\in M$. Since $x\preceq y$, $0\leq x\leq y$.
			\item It is easy to see that $0$ is a lower bound of $M$ with respect to the spectral order. If $y$ is the infimum of $M$, then $0\preceq y$. Hence $0\leq y$.
			\item Let $M\subseteq B_1(\Mcal_{sa})$ is nonempty. The set $M$ is bounded below by $-\unit$ and bounded above by \unit{} in $(\Mcal_{sa},\preceq)$. This ensures that there are the supremum of $M$, say $y$, and the infimum of $M$, say $z$, in $(\Mcal_{sa},\preceq)$. Clearly, 
			$$
			-\unit\preceq y\preceq x\preceq z\preceq \unit.
			$$
			Thus $y,z\in B_1(\Mcal_{sa})$.
			
			The proof of the case $L=\Ecal(\Mcal)$ follows easily from what we have proved above.
			
			Finally, assume that $M\subseteq P(\Mcal)$ is nonempty. As $M$ is bounded in $(\Mcal_{sa},\preceq)$ (bounded below by 0 and bounded above by \unit{}), the supremum $y$ and the infimum $z$ of $M$ exist. It is easy to see from Proposition~\ref{Well known results} that
			$$
			E^y_\la=\begin{cases}0,& \la<0;\\ \inf_{p\in M}(\unit-p),& \la\in[0,1);\\ \unit,& \la\geq 1;
			\end{cases}
			$$
			and 
			$$
			E^z_\la=\begin{cases}0,& \la<0;\\ \sup_{p\in M}(1-p),& \la\in[0,1);\\ \unit,& \la\geq 1.
			\end{cases}
			$$
According to \cite[Proposition~5.10]{Sch12}, spectra of $y$ and $z$ are contained in $\{0,1\}$. Thus $y,z$ are projections.
	\end{enumerate}
	\end{proof}

By the symbol $X\sqsubseteq Y$ we denote the fact that $X$ is a conditionally complete sublattice of a conditionally complete lattice $Y$. The previous proposition shows that
\vspace{0.5cm}
$$
\begin{tikzpicture}
	\node at (0,0) {$(P(\Mcal),\preceq)\sqsubseteq(\Ecal(\Mcal),\preceq)$};
	\node at (2.5,0.5) {\rotatebox{45}{$\sqsubseteq$}};
	\node at (4,0.75) {$(B_1(\Mcal_{sa}),\preceq)$};
	\node at (2.5,-0.5) {\rotatebox{-45}{$\sqsubseteq$}};
	\node at (4,-0.75) {$(\Mcal_{+},\preceq)$};
	\node at (5.5,0.5) {\rotatebox{-45}{$\sqsubseteq$}};
	\node at (5.5,-0.5) {\rotatebox{45}{$\sqsubseteq$}};
	\node at (6.5,0) {$(\Mcal_{sa},\preceq)$};
\end{tikzpicture}
$$
Moreover, it is easy to observe that $(B_1(\Mcal_{sa}),\preceq)$ and $(\Ecal(\Mcal),\preceq)$ are complete lattices and $(P(\Mcal),\preceq)$ is a complete sublattice of $(\Ecal(\Mcal),\preceq)$. Because there is no danger of confusion, we shall write $\sup_{x\in M} x$ (resp. $\inf_{x\in M} x$) for the supremum (resp. infimum) of $M$ with respect to spectral order without specifying of a concrete lattice in which we compute suprema and infima. 

\begin{lem}\label{bounded sets}
		A set $M\subseteq \Mcal_{sa}$ is order bounded with respect to the spectral order \ifff{} $M$ is norm bounded.
\end{lem}
	\begin{proof}
			Let $u,v\in\Mcal_{sa}$ be such that $u\preceq x\preceq v$ for all $x\in M$. Take $x\in M$. Then 
				$$
				-\max\{\norm{u},\norm{v}\}\unit\leq u\leq x\leq v\leq \max\{\norm{u},\norm{v}\}\unit.
				$$
				Therefore, $\norm{x}\leq \max\{\norm{u},\norm{v}\}$.
				
				For the converse, set $u=-K\unit$ and $v=K\unit$, where $K\geq 0$ satisfies $\norm{x}\leq K$ for all $x\in M$. Let $x\in M$. Then $E^x_\la=0$ when $\la<-K$ and $E^x_\la=\unit$ when $\la\geq K$. Hence $E^v_\la\leq E^x_\la\leq E^u_\la$ for all $\la\in\rr$. Thus $u\preceq x\preceq v$.
	\end{proof}
	
In the next theorem, we shall see that suprema and infima of increasing nets bounded above and decreasing nets bounded below, respectively, in $(\Mcal_{sa},\preceq)$ correspond to the strong operator limit. This implies, for example, that every decreasing net with a lower bound in $(\Mcal_{sa},\preceq)$ has the infimum in $(\Mcal_{sa},\preceq)$ which is equal to the infimum in $(\Mcal_{sa},\leq)$. In the sequel, the symbol $\slim_{\al\in \Lambda} x_\al$ will be denoted the strong operator limit of a net $(x_\al)_{\al\in\Lambda}$ in a \vNa{}.

\begin{theo}\label{Vigier type theorem}
Let \Mcal{} be a \vNa{}.
	\begin{enumerate}
		\item If $(x_\al)_{\al\in\Lambda}$ is a decreasing net in $(\Mcal_{sa},\preceq)$ with a lower bound, then 
		$$
		\inf_{\al\in\Lambda} x_{\al}=\slim_{\al\in\Lambda}	x_{\al}.
		$$
		\item If $(x_\al)_{\al\in\Lambda}$ is an increasing net in $(\Mcal_{sa},\preceq)$ with an upper bound, then 
		$$
		\sup_{\al\in\Lambda} x_{\al}=\slim_{\al\in\Lambda}	x_{\al}.
		$$
	\end{enumerate}
\end{theo}
		\begin{proof}\hfill{}
			\begin{enumerate}
				\item Set $x=\inf_{\al\in\Lambda} x_\al$. According to Proposition~\ref{Well known results}, $(x_\al)_{\al\in\Lambda}$ is decreasing and bounded below by $x$ in $(\Mcal_{sa},\leq)$. By Vigier's theorem, $y=\slim_{\al\in\Lambda} x_{\al}$ is the infimum of $(x_\al)_{\al\in\Lambda}$ in $(\Mcal_{sa},\leq)$ and so $x\leq y$.
				
			It remains to show the opposite inequality. Let $f:\rr\to\rr$ be a continuous increasing function. By Proposition~\ref{Well known results}, the net $(f(x_\al))_{\al\in\Lambda}$ is decreasing and bounded below in $(\Mcal_{sa},\leq)$. Therefore, $\slim_{\al\in\Lambda} f(x_\al)$ is the infimum of $(f(x_\al))_{\al\in\Lambda}$ in $(\Mcal_{sa},\leq)$. Let $\alpha_0\in\Lambda$. Consider $\Gamma=\set{\al\in\Lambda}{\al_0\leq \al}$. The net $(x_\al)_{\al\in\Gamma}$ is order bounded with respect to the spectral order and so it is norm bounded by Lemma~\ref{bounded sets}. As every continuous real-valued function on \rr{} is strong operator continuous on norm bounded subsets of $\Mcal_{sa}$ (see \cite[Proposition~5.3.2]{Ka97I}),
			$$
			f(y)=\slim_{\al\in\Gamma} f(x_\al)=\slim_{\al\in\Lambda} f(x_\al)\leq f(x_\beta)
			$$
for all $\beta\in\Lambda$. Since $f:\rr\to\rr$ was an arbitrary increasing continuous function, it follows from Proposition~\ref{Well known results} that $y\preceq x_\al$ for all $\al\in\Lambda$. Thus $y\preceq x$ which implies $y\leq x$.
			
				\item Since the function $t\mapsto -t$ is order-reversing, it follows from (i) that
					$$
					\sup_{\al\in\Lambda} x_{\al}=-\inf_{\al\in\Lambda} \left(-x_{\al}\right)=\slim_{\al\in\Lambda}	x_{\al}.
					$$
			\end{enumerate}
		\end{proof}
		
\begin{lem}\label{translation of supremum and infimum}
	Let \Mcal{} be a \vNa{}. Suppose that $\al\in\rr$ and $M\subseteq\Mcal_{sa}$ is nonempty. 
		\begin{enumerate}
			\item If $M\subseteq\Mcal_{sa}$ is bounded above, then $\sup_{x\in M}\left(x+\al\unit\right)=\left(\sup_{x\in M}x\right)+\al\unit$.
			\item If $M\subseteq\Mcal_{sa}$ is bounded below, then $\inf_{x\in M}\left(x+\al\unit\right)=\left(\inf_{x\in M}x\right)+\al\unit$.
		\end{enumerate}
\end{lem}
	\begin{proof}
	Statements are simple consequences of Proposition~\ref{Well known results}.
	\end{proof}
	
\begin{cor}\label{supremum}
	Let \Mcal{} be a \vNa{} and let $M\subseteq\Mcal_{sa}$ be a nonempty set bounded above. Assume that 
	$$
	\Fcal=\set{F\subseteq M}{F\mbox{ finite and nonempty}}.
	$$ 
	If $\delta_F\in(-\infty,\Delta_F]$, where $\Delta_F=\min_{x\in F}\inf\set{\la\in\rr}{E^x_\la\neq 0}$, for every $F\in \Fcal$, then
	\begin{eqnarray*}
	\sup_{x\in M}x=\slim_{F\in\Fcal}\left[\delta_F\unit
									+\slim_{n\in\nn}\left(\sum_{x\in F}\left(x-\delta_F\unit\right)^n\right)^{\frac{1}{n}}\right].
	\end{eqnarray*}
\end{cor}
	\begin{proof}
Let $F\subseteq\Fcal$. It is easy to see that $\Delta_F$ is a well defined real number. Take $\delta_F\in(-\infty,\Delta_F]$. If $x\in F$, then $x-\delta_F\unit\geq 0$ because $E^{x-\delta_F\unit}_\la=E^{x}_{\la+\delta_F}$ for every $\la\in\rr$. Applying Lemma~\ref{translation of supremum and infimum} and \cite[Theorem]{Ka79},
$$
\sup_{x\in F} x=\delta_F\unit+\sup_{x\in F}(x-\delta_F\unit)=\delta_F\unit
									+\slim_{n\in\nn}\left(\sum_{x\in F}\left(x-\delta_F\unit\right)^n\right)^{\frac{1}{n}}.
$$
It can be easily verified by a direct computation that $\sup_{x\in M} x=\sup_{F\in\Fcal} \sup_{x\in F} x$. Consequently, Theorem~\ref{Vigier type theorem} establishes the desired conclusion.
	\end{proof}
	
If the set $M$ from the previous corollary is bounded, we can choose $\delta_F$ independently of the choice of $F\in\Fcal$. Indeed, if $u$ is a lower bound of $M$, then we can take $\delta_F=-\norm{u}$ for all $F\in\Fcal{}$. In particular, if $M$ contains only positive elements, then we have 
$$
\sup_{x\in M}x=\slim_{F\in\Fcal}\slim_{n\in\nn}\left(\sum_{x\in F}x^n\right)^{\frac{1}{n}}
$$
which is a natural generalization of the formula from \cite[Theorem]{Ka79}.
	
\begin{cor}\label{infimum}
Let \Mcal{} be a \vNa{} and let $M\subseteq\Mcal_{sa}$ be a nonempty set with a lower bound $u$. Assume that 
	$$
	\Fcal=\set{F\subseteq M}{F\mbox{ finite and nonempty}}.
	$$ 
	If $\delta\in\rr$ is such that $u+\delta\unit$ is a positive invertible element, then 
	\begin{eqnarray*}
	\inf_{x\in M}x=-\delta\unit+\slim_{F\in\Fcal}\slim_{n\in\nn}\left(\sum_{x\in F}\left(x+\delta\unit\right)^{-n}\right)^{-\frac{1}{n}}.
	\end{eqnarray*}
\end{cor}
	\begin{proof}
	The set $M+\delta\unit$ consists of positive invertible elements because its lower bound $u+\delta \unit$ is a positive invertible element. Therefore, the infimum of every nonempty subset of $M+\delta\unit$ exists and is a positive invertible element. From Theorem~\ref{Vigier type theorem} and the fact that the inverse operation is order-reversing for the spectral order (see \cite{An89,FNN98}), we have
$$
\inf_{x\in M}(x+\delta \unit)=\inf_{F\in \Fcal}\inf_{x\in F}(x+\delta \unit)=\slim_{F\in\Fcal}\inf_{x\in F}(x+\delta \unit)=\slim_{F\in\Fcal}\left[\sup_{x\in F}(x+\delta \unit)^{-1}\right]^{-1}.
$$
By \cite[Theorem]{Ka79},
$$
\inf_{x\in M}(x+\delta \unit)=\slim_{F\in\Fcal}\left[\slim_{n\in\nn}\left(\sum_{x\in F}(x+\delta \unit)^{-n}\right)^{\frac{1}{n}}\right]^{-1}.
$$

Let $F\in\Fcal$ be arbitrary. If $n\in\nn$, then the function $t\mapsto t^{\frac{1}{n}}$ is operator concave on $[0,\infty)$ (see \cite[Proposition 1.3.11]{Pe79}). Applying Jensen's operator inequality \cite[Theorem~2.1]{HP03}, we see that there is $\al>0$ such that 
$$
\left(\sum_{x\in F}(x+\delta \unit)^{-n}\right)^{\frac{1}{n}}\geq |F|^{\frac{1}{n}}\sum_{x\in F} \frac{(x+\delta \unit)^{-1}}{|F|}\geq \frac{1}{|F|}\sum_{x\in F} (x+\delta \unit)^{-1}\geq \al\unit
$$
for all $n\in \nn$. By \cite[Theorem~5.3.4]{Ka97I}, the function $t\mapsto \frac{1}{\al+|t-\al|}$ is strong operator continuous on $\Mcal_{sa}$.
Hence
$$
 \inf_{x\in M} x
=-\delta\unit+\inf_{x\in M}(x+\delta \unit)=-\delta\unit+\slim_{F\in\Fcal}\slim_{n\in\nn}\left(\sum_{x\in F}(x+\delta \unit)^{-n}\right)^{-\frac{1}{n}}.
$$
	\end{proof}

Note that the previous result generalizes the formula for the infimum of two positive invertible elements in terms of the harmonic mean \cite{An89,FNN98}. Indeed, if $M$ is a nonempty finite set of positive invertible elements, then there is $\al>0$ such that $\al\unit\preceq x$ for all $x\in M$. Therefore, 
	$$
	\inf_{x\in M}x=\slim_{n\in\nn}\left(\sum_{x\in M}x^{-n}\right)^{-\frac{1}{n}}=\slim_{n\in\nn}\left(\sum_{x\in M}\frac{x^{-n}}{|M|}\right)^{-\frac{1}{n}}.
	$$

The next result says that the suprema and infima of a nonempty set of self-adjoint operators does not depend on underlying \vNa{}.

\begin{lem}\label{independence}
 Let \Mcal{} and \Ncal{} be \vNa{}s with the same unit and let $M$ be a nonempty subset of $\Mcal_{sa}$ as well as $\Ncal_{sa}$. 
	\begin{enumerate}
		\item If $M$ has an upper bound in $(\Mcal_{sa},\preceq)$, then both the supremum of $M$ in $(\Mcal_{sa},\preceq)$ and the supremum of $M$ in $(\Ncal_{sa},\preceq)$ exist and coincide.
		\item If $M$ has a lower bound in $(\Mcal_{sa},\preceq)$, then both the infimum of $M$ in $(\Mcal_{sa},\preceq)$ and the infimum of $M$ in $(\Ncal_{sa},\preceq)$ exist and coincide.
	\end{enumerate}
\end{lem}
	\begin{proof}\hfill{}
		\begin{enumerate}
			\item Spectral families of $x\in M$ in \Mcal{} and \Ncal{} are equal. Since $\inf_{x\in M}E^x_\la$ in $(P(\Mcal),\leq)$ coincide with the $\inf_{x\in M}E^x_\la$ in $(P(\Ncal),\leq)$, the desired result follows from Proposition~\ref{Well known results}.
			\item This follows from (i) by considering the order-reversing map $t\mapsto-t$.
		\end{enumerate}
	\end{proof}

Recall that self-adjoint operators $x$ and $y$ are said to be {\it orthogonal} if $xy=0$.

	\begin{cor}
	Let $M$ be a nonempty set of mutually orthogonal self-adjoint elements of a \vNa{} $\Mcal$. Suppose that 
	$$
	\Fcal=\set{F\subseteq M}{F\mbox{ finite and nonempty}}.
	$$ 	
	If $M$ is bounded above and it is not a singleton, then
	$$
	\sup_{x\in M}x=\slim_{F\in\Fcal}\sum_{x\in F}x^+,
	$$
where $x^+$ is the positive part of $x$.
\end{cor}
	\begin{proof}
By orthogonality of operators in $M$, the \vNa{} $\Ncal$ generated by $M\cup\{\unit\}$ is abelian. Let $F\in\Fcal{}$ contain at least two elements. It is easy to show by considering a \Ca{} $C(X)$ \Sic{} to \Ncal{} that $\sup_{x\in F}x=\sum_{x\in F}x^+$ in $(\Ncal{},\preceq)=(\Ncal{},\leq)$. According to Lemma~\ref{independence}, $\sup_{x\in F}x=\sum_{x\in F}x^+$ in $(\Mcal{},\preceq)$. Using Theorem~\ref{Vigier type theorem}, 
$$
\sup_{x\in M}x=\sup_{F\in \Fcal{}}\sup_{x\in F}x=\slim_{F\in\Fcal}\sum_{x\in F}x^+.
$$
	\end{proof}
	
Note that if $M$ is a set of at least two mutually orthogonal self-adjoint operators and is bounded below in $(\Mcal_{sa},\preceq)$, then $\inf_{x\in M} x=-\slim_{F\in\Fcal}\sum_{x\in F} x^-$, where $x^-$ is the negative part of $x$, because $t\mapsto -t$ is order-reversing.

\section{Order topology}

\begin{theo}\label{order convergence and spectral order}
Let \Mcal{} be a \vNa{}. If a net $(x_\al)_{\al\in\Lambda}$ is order convergent to $x$ in $(\Mcal_{sa},\preceq)$, then $(x_\al)_{\al\in\Lambda}$ is order convergent to $x$ in $(\Mcal_{sa},\leq)$.
\end{theo}
	\begin{proof}
Let $(y_\al)_{\al\in\Lambda}$ and $(z_\al)_{\al\in\Lambda}$ be nets in $\Mcal_{sa}$ such that $y_\al\preceq x_\al\preceq z_\al$ for all $\al\in \Lambda$, $(y_\al)_{\al\in\Lambda}$ is increasing net with the supremum $x$ in $(\Mcal_{sa},\preceq)$ and $(z_\al)_{\al\in\Lambda}$ is decreasing net with the infimum $x$ in $(\Mcal_{sa},\preceq)$. 
By Proposition~\ref{Well known results}, nets $(y_\al)_{\al\in\Lambda}$ and $(z_\al)_{\al\in\Lambda}$ are increasing and decreasing in $(\Mcal_{sa},\leq)$, respectively, and $y_\al\leq x_\al\leq z_\al$ for all $\al\in \Lambda$. It follows from Theorem~\ref{Vigier type theorem} and classical Vigier's theorem that $x$ is the supremum of $(y_\al)_{\al\in\Lambda}$ and the infimum of $(z_\al)_{\al\in\Lambda}$ in $(\Mcal_{sa},\leq)$.
	\end{proof}

Let $\Mcal_*$ be the predual of a \vNa{} \Mcal{}. We denote by $s(\Mcal,\Mcal_*)$ and $\tau(\Mcal,\Mcal_*)$ the $\sigma$-strong topology and the Mackey topology on \Mcal{}, respectively.

\begin{cor}\label{inclusions between topologies}
If \Mcal{} is a \vNa{}, then	
$$
s(\Mcal,\Mcal_*)|_{\Mcal_{sa}}\subseteq \tau(\Mcal,\Mcal_*)|_{\Mcal_{sa}}\subseteq \tau_o(\Mcal_{sa},\leq)\subseteq\tau_o(\Mcal_{sa},\preceq).
$$
\end{cor}
	\begin{proof}
Since the order topology is the finest topology preserving the order convergence, it follows from Theorem~\ref{order convergence and spectral order} that $\tau_o(\Mcal_{sa},\leq)\subseteq\tau_o(\Mcal_{sa},\preceq)$. It was proved in \cite{ChHW15} that $\tau(\Mcal,\Mcal_*)|_{\Mcal_{sa}}\subseteq \tau_o(\Mcal_{sa},\leq)$. The remaining inclusion is well known.
	\end{proof}

\begin{pro}\label{closed sets}
	If \Mcal{} is a \vNa{}, then the sets $\Mcal_+$, $B_1(\Mcal_{sa})$, $\Ecal(\Mcal)$, and $P(\Mcal)$ are closed in $\tau_o(\Mcal_{sa},\preceq)$.
\end{pro}
	\begin{proof}
The strong operator topology is weaker than $s(\Mcal,\Mcal_*)$ and so its restriction to $\Mcal_{sa}$ is weaker than $\tau_o(\Mcal_{sa},\preceq)$ by Corollary~\ref{inclusions between topologies}. Since $\Mcal_+$, $B_1(\Mcal_{sa})$, $\Ecal(\Mcal)$, and $P(\Mcal)$ are closed in the strong operator topology, we see that they are closed in $\tau_o(\Mcal_{sa},\preceq)$.
	\end{proof}
	
\begin{cor}\label{restrictions of order topology}
Let \Mcal{} be a \vNa{}. Then the following statements hold:
	\begin{enumerate}
		\item $\tau_o(\Mcal_{sa},\preceq)|_{\Mcal_+}=\tau_o(\Mcal_+,\preceq)$.
		\item $\tau_o(\Mcal_{sa},\preceq)|_{B_1(\Mcal_{sa})}=\tau_o(B_1(\Mcal_{sa}),\preceq)$.
		\item $\tau_o(\Mcal_{sa},\preceq)|_{\Ecal(\Mcal)}=\tau_o(\Mcal_{+},\preceq)|_{\Ecal(\Mcal)}=\tau_o(B_1(\Mcal_{sa}),\preceq)|_{\Ecal(\Mcal)}=\tau_o(\Ecal(\Mcal),\preceq)$.
		\item \begin{eqnarray*}
						\tau_o(\Mcal_{sa},\preceq)|_{P(\Mcal)}&=&\tau_o(\Mcal_{+},\preceq)|_{P(\Mcal)}=\tau_o(B_1(\Mcal_{sa}),\preceq)|_{P(\Mcal)}\\
																									&=&\tau_o(\Ecal(\Mcal),\preceq)|_{P(\Mcal)}=\tau_o(P(\Mcal),\preceq).
					\end{eqnarray*}
	\end{enumerate}
\end{cor}
	\begin{proof}
Since every conditionally complete lattice is a Dedekind complete poset, the corollary is immediately obtained by combining Proposition~\ref{suprema and infima in various posets}, Proposition~\ref{closed sets}, and \cite[Proposition~2.3]{ChHW15}.
	\end{proof}

\section*{Acknowledgement}
This work was supported by the project OPVVV Center for Advanced Applied Science CZ.02.1.01/0.0/0.0/16\_019/0000778 and  the ``Grant Agency of the Czech Republic" grant number 17-00941S, ``Topological and geometrical properties of Banach spaces and operator algebras II". 


\end{document}